\documentclass[11pt]{amsart}

\usepackage{amssymb}
\usepackage{amsthm}
\usepackage{amsmath}
\usepackage{graphicx}
\usepackage{amsaddr}
\usepackage{comment}
\usepackage[usenames,dvipsnames]{xcolor}
\usepackage[margin = 1in]{geometry}
\usepackage{enumerate}
\usepackage[toc,page]{appendix}
\usepackage{lineno}
\usepackage{color}
\usepackage{bbm}
\usepackage{hyperref}
\usepackage{mhchem}
\usepackage{siunitx}

\definecolor{mygreen}{rgb}{0.1,0.75,0.2}

\providecommand{\bbs}[1]{\left(#1\right)}

 \newtheorem{thm}{Theorem}[section]
 
 \newtheorem{lem}[thm]{Lemma}
 \newtheorem{prop}[thm]{Proposition}

 \numberwithin{equation}{section}

\providecommand{\bbs}[1]{\left(#1\right)}

\newcommand{\pt}{\partial}
\newcommand{\eps}{\varepsilon}
\newcommand{\ud}{\,\mathrm{d}}
\newcommand{\8}{\infty}

\newcommand{\ue}{u_{\eps}}

\newcommand{\bR}{\mathbb{R}}
\newcommand{\bZ}{\mathbb{Z}}

\newcommand{\R}{\mathbb{R}}
\newcommand{\N}{\mathbb{N}}

\begin{document}

\title[Slow Patterns in Multilayer Dislocation Evolution]{Slow Patterns in Multilayer Dislocation Evolution with Dynamic Boundary Conditions}

\author[Y. Gao]{Yuan Gao}
\address{Department of Mathematics, Purdue University, West Lafayette, IN 47906, USA}
\email{gao662@purdue.edu}

\author[S. Patrizi]{Stefania Patrizi}
\address{Department of Mathematics, The University of Texas at Austin, Austin, TX 78751, USA}
\email{spatrizi@math.utexas.edu}


\date{\today}

\begin{abstract}
In this paper, we study the slow patterns of multilayer dislocation dynamics modeled by a multiscale parabolic equation  in the half-plane  coupled with a dynamic boundary condition on the interface. We   focus on the influence of bulk dynamics with various relaxation time scales, on the slow motion pattern on the interface governed by an ODE system. Starting from a superposition of N stationary transition layers, at a specific time scale for the interface dynamics, we   prove that the dynamic solution approaches the superposition of N  explicit transition profiles whose centers solve   the  ODE system with a repulsive force. Notably, this ODE system is identical to the one obtained in the slow motion patterns of the one-dimensional fractional Allen–Cahn equation, where the elastic bulk is assumed to be static. Due to the fully coupled bulk and interface dynamics, new corrector functions with delicate estimates are constructed to stabilize the bulk dynamics and characterize the limiting behavior  of the dynamic solution throughout the entire half-plane. 
\end{abstract}

\maketitle

\section{Introduction}
~~

In this paper, we study the slow patterns of dislocation dynamics modeled by a multiscale parabolic equation coupled with a dynamic boundary condition.
Precisely, we aim to characterize the  asymptotic behavior when $\eps \to 0^+$ of the solution $\ue=\ue(x,y,t)$, $x\in\R$, $y\geq 0$, $t\geq0$, to
\begin{equation}\label{maineq}
\begin{cases}
\eps^a\pt_t \ue -  \Delta \ue=0,& y>0,\,t>0,\\
\eps \pt_t \ue -  \pt_y \ue + \frac{1}{\eps}W'(\ue)=0, &y=0,\,t>0,\\
\ue=u_\eps^0,&t=0,
\end{cases}
\end{equation}
 where $a>0$ and $W$ is a multi-well potential. In the following, we will introduce the background, specific setup, main results and approaches.
 
 \subsection{Background and motivations} Dislocations, which are  line defects in crystalline materials,  play a crucial role  in the study of mechanical
behaviors of materials.  In particular,  the motion of dislocations may cause fatal plastic deformations. To unveil the core structure of dislocations
  --  small regions of heavily distorted atomistic structures   --    the Peierls-Nabarro (PN) model introduced by \textsc{Peierls and Nabarro} \cite{Peierls, Nabarro}  is
a multiscale continuum model for displacement $\mathbf{u}$.   It    incorporates the atomistic effect by introducing a nonlinear
potential $W$ 
  that describes  the atomistic misfit interaction across the   dislocation’s slip plane, 
while the 
  elastic continua associated with elastic energy are connected by the interface misfit potential
$W$. 

Based on a simplified two-dimensional PN model, where the displacement $\mathbf{u}$ is replaced by a scalar variable $u$, our goal is to study the relaxation pattern of the dislocation dynamics with particular focus on the influence of the bulk dynamics with different time scaling. Precisely, consider
 the total free energy  
 \begin{equation}
 E_\eps(u) = \frac12 \int_{\bR^2_+} \eps |\nabla u|^2 \ud x\ud y+ \int_\Gamma W(u) \ud x.
 \end{equation}
Here the first term represents   the  elastic energy in the bulk $\R^2_+:=\R\times(0,\infty)$, and the second term represents the interface misfit energy on 
  the  slip plane $\Gamma:=\{(x,y); \,\, x\in \bR, y=0\}.$
Consider the simplest quadratic Rayleigh dissipation functional  including frictions in the bulks and on the interface as the dissipation metric
$$g(\dot{u},\dot{v})= \eps^{a+1}  \int_{\bR^2_+}  \dot{u} \dot{v} \ud x\ud y+ \eps^2 \int_{\Gamma}  \dot{u}\dot{v} \ud x.$$
Notice the dissipation scaling for the interface is fixed to be $\eps^2$, but the dissipation scaling for the bulk is $\eps^{a+1}$ with a parameter $a>0$. 
 Then,  the gradient flow of $E_\eps(u)$ with respect to   the   dissipation metric $g$ is determined by, tested with   any virtual velocity $\dot{u}$,  
 \begin{equation}\label{gf}
\begin{aligned}
g(\pt_t u , \dot{u}) =&  - \frac{\ud}{\ud \delta}\Big|_{\delta=0} E(u+\delta \dot{u}) \\
=&  - \int_{\mathbb{R}^2_+} \eps \nabla u \nabla \dot{u} \ud x \ud y  - \int_\Gamma W'(u) \dot{u} \ud x 
= \int_{\mathbb{R}^2_+} \eps \Delta u \dot{u} \ud x \ud y  +\int_{\Gamma}[\eps \pt_y u - W'(u)] \dot{u} \ud x.
\end{aligned}
\end{equation}
Taking an arbitrary virtual velocity $\dot{u}$, we obtain the governing equation \eqref{maineq}.

    When $a\to +\8$, the bulk dynamics relaxes very fast to   a  stationary state and thus, by the Dirichlet to Neumann map for the Laplacian,   see \cite{Caffarelli_Silvestre_2007}, we have   
    $\pt_y u(x,0)  =-(-\Delta)^{\frac12}{u} (x,0)$,    and  the full dynamics reduces to the classical one-dimensional fractional Allen-Cahn equation (a  nonlocal reaction-diffusion equation) 
\begin{equation}\label{AC1d}
\eps \pt_t \ue + (-\Delta)^{\frac12} \ue + \frac{1}{\eps}W'(\ue)=0.
\end{equation}
For \eqref{AC1d}  at this specific time scaling, the slow motion of a multilayer profile  guided by an ODE system is studied   by  \textsc{ Gonzalez and Monneau}  in \cite{Mon1}. This is also the reason we fixed the dissipation scaling on the interface to be $\eps^2$.  More precisely, 
in  \cite[Theorem 1.1]{Mon1} it is proven that the solution $\ue $ to \eqref{AC1d},   with a  well-prepared initial datum -- a superposition of $N$ transition layers (see \eqref{initialcondition}) -- approaches as $\eps\to0^+$, and for every fixed time,  the integers $1,2,\ldots, N$,  and the jump points, $z_i(t)$,  between two consecutive integers, move accordingly to the following 
ODE system, for $i=1,\ldots, N$, 
\begin{equation} \label{PN} 
   \left\{ \begin{aligned}
     &\frac{dz_i}{dt} = \frac{c_0}{\pi} \sum_{ j \neq i}\frac{1}{z_i - z_j},
     && t>0, \\
     &z_i(0) = z_i^0. 
   \end{aligned} \right.
\end{equation}
Here $z_i^0$ is the center of each transition layer at initial time  and  $c_0>0$ is defined in \eqref{eq:c0-gamma}.  The above slow motion pattern for multilayer profiles was previously studied for  the classical one-dimensional Allen-Cahn equation \cite{Pego, Chen_2004}. In the nonlocal case, equation \eqref{AC1d}  where the operator $ (-\Delta)^{\frac12} $ is replaced by 
$(-\Delta)^{s}$,   $s\in(0,1)$,  and with the appropriate space scaling,  is studied in \cite{DFV, DPV}. In \cite{PV15}, more general initial data, including possible opposite orientations of dislocations,  were considered and the slow motion pattern before a finite collision time was proven to be driven by a similar ODE system including either repulsive or attractive particle interactions. We refer to \cite{PV1, PV2} for comprehensive study of the long time behaviors of those multilayer dislocation profiles including possible finite time collisions. The case where $N\to\infty$ is studied in \cite{PSan1, PSan2}.
  Properties of the ODE system \eqref{PN}  have been studied in   \cite{FIM09} and in the more general case in which collisions are allowed in \cite{van2022discrete}. 

  Beyond the reduced one-dimensional nonlocal dynamics,  the long time behavior of the fully coupled bulk-interface dynamics \eqref{maineq} with fixed $\eps=1$ and a single layer profile ($N=1$) was established in \cite{gao2023asymptotic}.  Then the natural question  is whether the same slow motion behavior of the multilayer dislocation profile, guided by the  ODE system \eqref{PN}, observed  in \eqref{AC1d},  can be obtained when the bulk dynamics,  with different dissipation scalings,  are coupled with the interface dynamics.

In this paper, we investigate  the full dynamics \eqref{maineq}, including the bulk dynamics with various relaxation scalings $\eps^{a}$,  and 
 prove that  the slow motion pattern driven by the ODE system \eqref{PN} can still be observed, particularly for  finite $a > 0$.  The most delicate case is when $a\leq 1$, as discussed  in Subsection \ref{subHeu}. 
 Furthermore, we explicitly characterize the limiting behavior of the solution for all $y\geq 0$. 
 
 In the following subsection, we outline the specific setting and present our main result. 

\subsection{Setting of the problem and main result}
Assume the nonconvex potential $W$  satisfies
\begin{equation}\label{eq:W}
\begin{cases}
W \in C^{2, \beta} (\bR) & \hbox{for some}~0 < \beta <1, \\
W(u+1) = W(u) & \hbox{for any}~ u \in \bR,\\
W= 0 & \hbox{on}~\bZ,\\
W>0 & \hbox{on}~\bR \setminus \bZ,\\
W''(0) >0.
\end{cases}
\end{equation}
 
 Let $\ue$  be the solution to \eqref{maineq}  when the initial condition $\ue^0$  is a superposition of layer
solutions.
 The stationary layer solution (also called the phase transition) $\phi=\phi(x,y)$ is the unique solution to  
\begin{equation} \label{eq:standing wave}
\begin{cases}
-\Delta \phi=0,&y>0,\\
\pt_y\phi = W'(\phi), & y=0,\\
 \pt_x\phi>0,&\text{on }\R, \\
\phi(-\infty,y) = 0, \quad \phi(+\infty,y)=1,\quad \phi(0,0) = \frac{1}{2}.
\end{cases}
\end{equation}

  Then, for $z_1^0<z_2^0<\ldots<z_N^0$,  we set initial data as
\begin{equation}\label{initialcondition}u_\eps^0(x,y):=\sum_{i=1}^N \phi\left(\frac{x-z_i^0}{\eps},\frac{y}{\eps}\right).
\end{equation}

When a special periodic misfit potential is chosen  $$W'(u)=-\frac{1}{2\pi}\sin\left[2\pi\left(u-\frac12\right)\right],$$
it is well-known \cite{CS05, Mon1}  that 
\begin{equation}\label{Phi}\Phi(x,y)= \frac{1}{\pi}\arctan\left(\frac{x}{y+1}\right)+\frac12\end{equation}
is the only solution to \eqref{eq:standing wave}. 
Further discussions on $\phi$ will be presented in Section \ref{phisection}.


The following is the main result of our paper.
\begin{thm}\label{mainthm}
Assume \eqref{eq:W} and  let $\ue$ be the solution to   \eqref{maineq} with $a>0$. Assume the initial condition is given by \eqref{initialcondition}. Let 
\begin{equation}\label{v0def}v_0(x,t):=\sum_{i=1}^N H(x-z_i(t)),
\end{equation}
where $H$ is the Heaviside function and $(z_1(t),\ldots,z_N(t))$ is the solution to \eqref{PN}.  Then, as $\eps\to 0^+$,  $\ue$ exhibits the following asymptotic behavior:
\begin{itemize} 
\item[i)] For $t\geq0$, $x\in\R$, 
\begin{equation}\label{asymptotic1}
\limsup_{(x',y',t')\to(x,0,t) \atop \eps\to 0^+}\ue(x',y',t')\leq ( v_0)^*(x,t)\quad\text{and}\quad \liminf_{(x',y',t')\to(x,0,t)\atop \eps\to 0^+}\ue(x',y',t')\geq (v_0)_*(x,t),
\end{equation}
where $(v_0)_*$  and $( v_0)^*$  are the lower and upper semicontinuous envelopes of $v_0.$
\item[ii)]  For $t\geq0$, $x\in\R$, and  $y>0$, 
\begin{equation}\label{asymptotic2}
\lim_{\eps\to0^+}\ue(x,y,t)=\frac{1}{\pi}\sum_{i=1}^N \left(\frac{\pi}{2}+\arctan\left(\frac{x-z_i(t)}{y}\right)\right).
\end{equation}
\end{itemize}
\end{thm}

\bigskip

\subsection{Heuristics}\label{subHeu}

The main approach we use is to construct appropriate  super/subsolutions to  \eqref{maineq} and \eqref{initialcondition} as barrier functions. By applying  the comparison principle, the limiting behavior of the dynamic solution $u_\eps$ is dominated by those barrier functions. These barriers are  constructed from a formal ansatz  that satisfies  the equations and the initial condition,  up to small errors. 

For the equation \eqref{AC1d},   the ansatz   derived in \cite{Mon1} is given by
$$v_\eps(x,y,t):= \sum_{i=1}^N \left[ \phi\bbs{ \frac{x- z_i(t)}{\eps}, \frac{y}{\eps} } - \eps \dot z_i(t)\psi\bbs{ \frac{x- z_i(t)}{\eps}, \frac{y}{\eps} }\right], $$
with $y=0$, where $(z_1(t),\ldots,z_N(t))$ solves \eqref{PN}, $\phi$ is the layer solution given by \eqref{eq:standing wave},  and $\psi$ is a corrector that  controls the error of order 1 when substituting  the ansatz into  \eqref{AC1d}. 
The equation for $\psi$ is given in   \eqref{psi}. 

However, $v_\eps$  is not an appropriate ansatz for  \eqref{maineq} when $y>0$. Indeed, using that $\phi$ and $\psi$ are both harmonic for $y>0$,  we obtain upon substituting $v_\eps$  into the equation,  
$$\eps^a\pt_t v_\eps -  \Delta v_\eps=-\eps^{a-1} \sum_{i=1}^N \dot z_i(t) \partial_x\phi\bbs{ \frac{x- z_i(t)}{\eps}, \frac{y}{\eps} }+o_\eps(1),$$
where $o_\eps(1)\to0$ as $\eps\to0^+$. 
Thus,  the right-hand side of the equation is not small for $a\leq 1$. To address this issue,  we introduce an additional corrector $q$ that satisfies  
\begin{equation}\label{qeqintro}
\begin{cases}
-\Delta q(x,y) = \pt_x \phi(x,y), &(x,y)\in \bR^2_+,\\
q(x,0) = 0, & x\in \bR. 
\end{cases}
\end{equation}
We then  consider the new ansatz $w_\eps$ which is obtained by adding a lower-order correction to $v_\eps$, 
$$w_\eps(x,y,t):=v_\eps(x,y,t)+\eps^{a+1} \sum_{i=1}^N  \dot z_i(t) (t)q\bbs{ \frac{x- z_i(t)}{\eps} , \frac{y}{\eps}}.$$
Formally, when we  plug  this ansatz into \eqref{maineq} for $y>0$, we obtain 
\begin{equation}\label{Heristeq}\eps^a\pt_t w_\eps -  \Delta w_\eps=o_\eps(1). \end{equation}
However, due to the lack of integrability properties in the entire half-plane, we must replace  $ \pt_x \phi(x,y)$  by $ \pt_x \phi(x,y)g(y)$ in \eqref{qeqintro},  where $g$ is a cutoff function. 
Delicate growth estimates for the corrector $q$ are crucial to control the bulk dynamics in  the entire half-plane. These, along with new decay estimates for $\phi$ and $\psi$, are derived in Section \ref{sec3}. To control the error in \eqref{Heristeq}, additional terms must be added to $w_\eps$, see Section \ref{sec4}. 
Since, in \eqref{maineq}, the interface dynamics and the bulk dynamics affect each other in a two-way coupling manner through the dynamic boundary condition with a Neumann derivative, these additional terms are introduced in a suitable way. 
\subsection{Notations}

In the paper, we will denote by $C>0$ any  constant independent of $\eps$. 

For $\beta \in (0,1]$ and $k \in \N \cup \{0\}$, we denote by $C^{k,\beta}(\R)$ the usual class of functions with bounded $C^{k,\beta}$ norm over $\R$. 

We denote $\R^2_+:=\R\times(0,\infty)$.

Given a function $\eta = \eta(x,y,t)$, defined on a set $A$ of $\overline{\R^2_+}\times[0,\infty)$,  we write $\eta = O(\eps)$ if there is $C>0$ such that
$|\eta(x,y,t)| \leq C \eps$ for all $(x,y,t)\in A$.

For a set $A$, we denote by $\chi_A$ the characteristic function of  $A$.

  Given a function $v(x,y,t)$ we denote by $v_*$  and $v^*$  the lower and upper semicontinuous envelopes, respectively defined  by 
$$v_*(x,y,t):=\liminf_{(x',y',t')\to (x,y, t)} v(x',y', t'),$$
$$v^*(x,y, t):=\limsup_{(x',y',t')\to (x,y,t)} v(x',y',t').$$

The Heaviside function is defined by 
$$H(x):=\begin{cases}1,&\text{if }x>0,\\
0,&\text{if }x<0.
\end{cases}
$$
The explicit value at 0 that is assumed to be in $[0,1]$, plays no role. 
\subsection{Organization of the paper}

The rest of the paper is organized as follows. In Section \ref{sec2}, we review and prove some preliminary results on stationary solutions and corrector functions. In Section \ref{sec3}, we construct super- and subsolutions to the full dynamics \eqref{maineq} with initial condition \eqref{initialcondition}. In Section \ref{sec4}, we complete the proof of Theorem \ref{mainthm}.

 \section{Preliminary results}\label{sec2}
 In this section, we will first review and establish some preliminary results. The properties of the ODE system \eqref{PN} will be discussed in Section \ref{ODEsection}, while the decay estimates for the stationary layer solution $\phi(x,y)$  will be presented in Section \ref{phisection}. To prepare for the construction of super- and subsolutions in Section \ref{sec3}, we introduce two corrector functions, 
 $\psi(x,y)$, $q(x,y)$, and    prove some essential decay estimates for each in Section \ref{psisection} and Section \ref{qsection},  respectively.
\subsection{Preliminary results on the ODE system}\label{ODEsection}
We first recall the following lower bound estimate for the minimal particle distance in ODE \eqref{PN}.  
\begin{lem}\label{cdotlemma}\cite[Theorem 2.4]{van2022discrete}
Let $(z_1(t),\ldots,z_N(t))$ be the solution of \eqref{PN} and let 
$$d(t):=\min\{|z_i(t)-z_j(t)|,\,\, i\neq j, \, i,j=1,\cdots, N\}$$
be the minimal distance between dislocation points. Then 
$$d(t)\geq \sqrt{\frac{8}{N^2-1}t+d(0)^2}.$$
\end{lem}

\subsection{The layer solution  $\phi$}\label{phisection}

Next, we  summarize some properties of the layer function $\phi$, solution to \eqref{eq:standing wave}.
For convenience in the notation, let $c_0$ and $\alpha$ be given respectively by
\begin{equation}\label{eq:c0-gamma}
c_0^{-1} = \int_\R \left(\pt_x\phi(x,0)\right)^2 \, d x\quad \hbox{and} \quad \alpha = W''(0),
\end{equation}
and let $H(x)$ be the Heaviside function. 

\begin{lem}\label{lem:asymptotics}  
There is a unique solution $\phi \in C^{2, \beta}(\overline{\R^2_+})$ of \eqref{eq:standing wave}, with the same $\beta$ as in \eqref{eq:W}.
 Moreover, there exists a constant $C>0$  such that
\begin{equation}\label{eq:asymptotics for phi}
\left |\phi(x,0) - H(x) + \frac{1}{\alpha\pi x}\right| \leq \frac{C}{x^{2}}, \quad \text{for }|x|\geq 1,
\end{equation}
 and
\begin{equation}\label{eq:asymptotics for phi_x}
\frac{y+1}{C(x^2+(y+1)^2)}\leq \pt_x\phi(x,y) \leq \frac{C(y+1)}{x^2+(y+1)^2},  \quad \text{for all }(x,y)\in\overline{\R^2_+}. 
\end{equation}
In particular, 
\begin{equation}\label{eq:asymptotics for phi_xy=0}
\frac{1}{Cx^2}\leq \pt_x\phi(x,0) \leq \frac{C}{x^2},  \quad \text{for  }|x|\geq 1. 
\end{equation}

\end{lem}

\begin{proof}
Existence of a unique solution $\phi\in C^{2,\beta}(\overline{\R_2^+})$  of \eqref{eq:standing wave} is proven in \cite{CS05}, see  Theorem 1.2 and  Lemma 2.3.
For  estimate \eqref{eq:asymptotics for phi_x},   see Theorem 1.6 and formulas (6.16) and (6.18) in the same paper. 
Estimate \eqref{eq:asymptotics for phi} is  proven in \cite[Theorem 3.1]{Mon1}.
\end{proof}

\begin{lem}
Let $\phi$ be the solution of \eqref{eq:standing wave}, given by Lemma \ref{lem:asymptotics}. Then, there exists $C>0$ such that  for $\eps,\, y>0$,
\begin{equation}\label{phiasymptocsylarge}
\frac{1}{\pi}\left(\frac{\pi}{2}+\arctan\left(\frac{x- \eps^\frac12}{y}\right)\right)-C\eps^{\frac12}\leq\phi\left(\frac{x}{\eps},\frac{y}{\eps}\right)\leq \frac{1}{\pi}\left(\frac{\pi}{2}+\arctan\left(\frac{x+ \eps^\frac12}{y}\right)\right)+C\eps^{\frac12}.
\end{equation}
In particular, for $y>0$, 
$$\lim_{\eps\to0^+}\phi\left(\frac{x}{\eps},\frac{y}{\eps}\right)=\frac{1}{\pi}\left(\frac{\pi}{2}+\arctan\left(\frac{x}{y}\right)\right).$$
\end{lem}
\begin{proof}
Since $\phi$ is harmonic in $\R^2_+$, it can be written as the convolution of the Poisson kernel in the half-plane with $\phi(x,0)$. 
Thus for $\eps,\,y>0$, 
$$\phi\left(\frac{x}{\eps},\frac{y}{\eps}\right)=\frac{1}{\pi}\int_{\R}\phi\left(\xi,0\right)\frac{y/\eps}{(x/\eps-\xi)^2+(y/\eps)^2}\ud \xi
=\frac{1}{\pi}\int_{\R}\phi\left(\frac{\zeta}{\eps},0\right)\frac{y}{(x-\zeta)^2+y^2}\ud \zeta,$$
where we performed  the change of variable $\zeta=\eps\xi$.  By \eqref{eq:asymptotics for phi}, and using that $\phi<1$, we obtain
\begin{align*}
\frac{1}{\pi}\int_{\R}\phi\left(\frac{\zeta}{\eps},0\right)\frac{y}{(x-\zeta)^2+y^2}\ud \zeta
&\leq\frac{1}{\pi}\left(\int_{-\infty}^{-\eps^\frac12}\frac{-C\eps}{\zeta}+\int_{-\eps^\frac12}^\infty\right)
\frac{y}{(x-\zeta)^2+y^2}\ud \zeta\\&
\leq \frac{1}{\pi}\left(\int_{-\infty}^{-\eps^\frac12}C\eps^\frac12+\int_{-\eps^\frac12}^\infty\right)
\frac{y}{(x-\zeta)^2+y^2}\ud \zeta\\&
\leq C\eps^\frac12+ \frac{1}{\pi}\int_{-\eps^\frac12}^{\infty}\frac{y}{(x-\zeta)^2+y^2}\ud \zeta\\&
=  \frac{1}{\pi}\left(\frac{\pi}{2}+\arctan\left(\frac{x+\eps^\frac12}{y}\right)\right)+C\eps^{\frac12}. 
\end{align*}
 Therefore, we have 
$$\phi\left(\frac{x}{\eps},\frac{y}{\eps}\right)\leq \frac{1}{\pi}\left(\frac{\pi}{2}+\arctan\left(\frac{x+ \eps^\frac12}{y}\right)\right)+C\eps^{\frac12}.$$
Similarly, one can prove
$$\phi\left(\frac{x}{\eps},\frac{y}{\eps}\right)\geq \frac{1}{\pi}\left(\frac{\pi}{2}+\arctan\left(\frac{x- \eps^\frac12}{y}\right)\right)-C\eps^{\frac12}.$$
Estimate \eqref{phiasymptocsylarge} follows. 
\end{proof}

\subsection{The corrector $\psi$}\label{psisection}
We now introduce the first corrector $\psi$, which will be used later to control the interface dynamics.
As in  \cite{Mon1}, we define the function  $\psi$ to be the solution of
\begin{equation}\label{psi}
\begin{cases}
-\Delta \psi  =0, & y>0,\\
    \pt_y \psi= W''(\phi)\psi + \frac{1}{\alpha c_0}(W''(\phi) - W''(0))  + \partial_x\phi,  & y=0,\\
  \lim_{|x| \rightarrow \infty} \psi(x,0)=0,
\end{cases}
\end{equation} 
where $c_0,\,\alpha$ are defined in \eqref{eq:c0-gamma}. 
We will use $\psi$ as an $O(\eps)$ correction to construct sub and  supersolutions to \eqref{maineq}. 
For a detailed heuristic motivation of
equation \eqref{psi}
see  \cite[Section 3.1]{Mon1}. In the next lemma, we present some known results about the function 
$\psi$  as well as new estimates. 
\begin{lem}
\label{l:psi}
There exists a unique solution $\psi \in C_{loc}^{1,\beta}(\overline{\R^2_+}) \cap W^{1,\infty}(\overline{\R^2_+}) $ to \eqref{psi}. Furthermore, there exist constants $c \in \R$
and $C > 0$  such that
\begin{equation} \label{psiinfinity}
\left|\psi(x,0)-\frac{c}{
x}\right|\leq\frac{C}{x^2},\quad|\pt_x\psi(x,0)|\leq \frac{C}{x^2} \quad\text{for }  |x|\geq 1,
\end{equation}
 
\begin{equation} \label{psiinfinityy>0}
| \psi(x,y)|\leq \frac{C}{y},\quad |\pt_x \psi(x,y)|\leq \frac{C}{y} \quad\text{for all }x\in\R\text{ and }y\geq1,
\end{equation}
and
\begin{equation} \label{psiinfinityx>1}
| \psi(x,y)|\leq \frac{C}{|x|}\quad\text{for all }x>1\text{ and }y\geq0.
\end{equation}

\end{lem}
\begin{proof}
Existence of a solution $\psi \in C_{loc}^{1,\beta}(\overline{\R^2_+}) \cap W^{1,\infty}(\overline{\R^2_+}) $  of \eqref{psi} is proven in \cite[Theorem 3.2]{Mon1}.
Estimates \eqref{psiinfinity} are proven in \cite[Lemma 3.2]{MonneauPatrizi2}.

Let us prove \eqref{psiinfinityy>0}. 
Let $\Phi$ be the explicit layer solution given by \eqref{Phi}. Then $\Phi$ satisfies  \eqref{eq:asymptotics for phi} with $\alpha=1$. 
In particular, for $a,\,b>0$, 
$$\Phi\left(\frac{x}{b},0\right)- \Phi\left(\frac{x}{a},0\right)\leq   \frac{a-b}{ \pi x }+\frac{C}{x^2+1}.$$
Choosing $a$ and $b$ such that $b-a=c$ with $c$ as in \eqref{psiinfinity}, we see that 
$$ \psi(x,0)\leq  \Phi\left(\frac{x}{a},0\right)- \Phi\left(\frac{x}{b},0\right)+\frac{C}{x^2+1}\leq \Phi\left(\frac{x}{a},0\right)- \Phi\left(\frac{x}{b},0\right)+C\partial_x\Phi(x,0).$$
Since the functions $\psi$, $\Phi$ and $\partial_x\Phi$ are all harmonic in $\R^2_+$, the comparison principle implies that, for all $(x,y)\in \overline{\R^2_+}$,
\begin{equation}\label{psiestimatePhi}\psi(x,y)\leq \Phi\left(\frac{x}{a},\frac{y}{a}\right)- \Phi\left(\frac{x}{b},\frac{y}{b}\right)+ C\partial_x\Phi(x,y).\end{equation}
Now, by the Mean Value Theorem, for some $\lambda\in (0,1)$,  
\begin{align*}
 \Phi\left(\frac{x}{a},\frac{y}{a}\right)- \Phi\left(\frac{x}{b},\frac{y}{b}\right)&=\frac{1}{\pi}\arctan\left(\frac{x}{y+a}\right)-\frac{1}{\pi}\arctan\left(\frac{x}{y+b}\right)\\&
 =\frac{1}{\pi}\frac{1}{1+s^2}_{|s=x\big(\frac{\lambda}{y+a}+\frac{1-\lambda}{y+b}\big)}x\left(\frac{1}{y+a}-\frac{1}{y+b}\right)\\&
 =\frac{b-a}{\pi}\frac{(y+a)(y+b)x}{(y+a)^2(y+b)^2+x^2(y+\lambda b+(1-\lambda)a)^2}.
\end{align*}
If $1\leq y\leq 2 \max\{a,b\}$, then 
\begin{align*}
 \Phi\left(\frac{x}{a},\frac{y}{a}\right)- \Phi\left(\frac{x}{b},\frac{y}{ b}\right)&\leq C\leq \frac{\tilde C}{y}.
 \end{align*}
If instead   $y> 2 \max\{a,b\}$, then 
 \begin{align*}
 \Phi\left(\frac{x}{a},\frac{y}{a}\right)- \Phi\left(\frac{x}{b},\frac{y}{ b}\right)&\leq\frac {C y^2|x|}{y^4+x^2y^2}=\frac {C |x|}{y^2+x^2}\leq \frac{\tilde C}{y}.
  \end{align*}
Combining the last two inequalities with \eqref{psiestimatePhi} yields, 
$$\psi(x,y) \leq \frac{C}{y}\quad\text{for }y\geq 1.$$
Similarly, one can prove 
$$\psi(x,y) \geq- \frac{C}{y}\quad\text{for }y\geq 1. $$ Estimate  \eqref{psiinfinityy>0} for $\psi$ follows. The same argument also gives estimate  \eqref{psiinfinityx>1}. 

Next, from \eqref{eq:asymptotics for phi_xy=0} and \eqref{psiinfinity}, there is $C>0$ such that $-C\pt_x\phi(x,0)\leq \pt_x\psi(x,0)\leq C\pt_x\phi(x,0)$. Since both $\pt_x\phi$ and $\pt_x\psi$ are harmonic in $\R^2_+$, by the comparison principle we get $-C\pt_x\phi(x,y)\leq \pt_x\psi(x,y)\leq C\pt_x\phi(x,y)$ for every $(x,y)\in\overline{\R^2_+}$,  which combined with \eqref{eq:asymptotics for phi_x} gives \eqref{psiinfinityy>0} for $\pt_x\psi$. 
 
\end{proof}

\subsection{The corrector $q$}\label{qsection}
We introduce a further corrector, $q$, solution to 
 \begin{equation}\label{qequation}
\begin{cases}
-\Delta q(x,y) = \pt_x \phi(x,y)g(y), &(x,y)\in \bR^2_+,\\
q(x,0) = 0, & x\in \bR,
\end{cases}
\end{equation}
where $g$ is a smooth cut-off function. 
We will use $q$ as an $O(\eps^{a+1})$ correction to control the bulk dynamics when  constructing sub and  supersolutions to \eqref{maineq}. 

Existence and properties of $q$ are proven in the next lemma. 
\begin{lem}\label{qlemma}
 Let $g(y)$ be a smooth nonnegative  function with support in $[0,R],$ $R>2$. 
 Then there exists a unique bounded solution $q$ of \eqref{qequation}, 
 where $\phi$ is the solution of  \eqref{eq:standing wave}. Moreover, there exists a constant $C>0$,  such that
 \begin{equation}\label{qestimates}
0\leq q(x,y)\leq CR\ln R, \quad \text{for all }(x,y)\in \overline{\bR^2_+},
 \end{equation} 
\begin{equation}\label{qderestimates}
|\pt_x q(x,y)|,\,|\pt_y q(x,y)|\leq C\ln R, \quad \text{for all }(x,y)\in \overline{\bR^2_+},
 \end{equation}
  and 
\begin{equation}\label{qderestimatesylarge}
q(x,y)\leq \frac{CR^2}{y},\quad 
|\pt_x q(x,y)|,\,|\pt_y q(x,y)|\leq \frac{CR}{y}, \quad \text{for all }x\in \R\text{ and }y\geq 2R. 
 \end{equation}  
 \end{lem}

\begin{proof}

 Consider the Green function in the half-plane, given by 
  $$G(Z',Z)=\frac{1}{2\pi}\left(\ln|Z'-\tilde Z|-\ln|Z'-Z|\right),$$
   where if $Z=(x,y)\in \bR^2_+$, then $\tilde Z=(x,-y)$.
   Define, 
   $$f(Z):= \pt_x \phi(Z)g(y),$$ 
     and 
   \begin{equation*}
q(Z):=\int_{\bR^2_+} G(Z', Z) f(Z') \ud Z'.   
\end{equation*}
  We will show that $q$ is well-defined and satisfies estimates \eqref{qestimates}-\eqref{qderestimatesylarge}.
In particular,  $q$ is  a  smooth solution of \eqref{qequation}.  The uniqueness is a consequence of the uniqueness of the bounded solution to  \eqref{qequation}.

Since for $Z,\,Z'\in\bR^2_+$, 
 $$|Z'-\tilde Z|\leq |Z'- Z|+|Z-\tilde Z|=|Z'- Z| +2y,$$ 
 we have that 
\begin{equation}\label{Greeninfinity0}
 0\leq \ln|Z'-\tilde Z|-\ln|Z'-Z|\leq \ln( |Z'- Z|+2y)-\ln (|Z'- Z|)=\ln\left(1+\frac{2y}{|Z'-Z|}\right).
 \end{equation}
 Moreover, by \eqref{eq:asymptotics for phi_x}, 
\begin{equation}\label{asymptioticphi_xg}
0\leq f(Z')\leq \frac{C(y'+1)}{x'^2+(y'+1)^2}\chi_{[0,R]}(y')\leq \frac{C}{y'+1}\chi_{[0,R]}(y'). 
\end{equation}
Since both $f$ and $G$ are nonnegative, nonzero functions, $q$ is positive. 
Let us show the upper bound for $q$ in \eqref{qestimates}.  In view of  \eqref{qderestimatesylarge}, we may assume that $y< 2R$.
 We write 
  \begin{equation}\label{qsplitlemmaqest}
\begin{aligned}
q(Z)& =\bbs{ \int_{\bR^2_+\cap\{|Z'-Z|<1\}} + \int_{\bR^2_+\cap\left\{|Z'-Z|>1\right\}}} 
G(Z', Z) f(Z') \ud Z'=:I_1+I_2.\\
\end{aligned}
\end{equation}
By  \eqref{Greeninfinity0}, and using that $\ln|Z'-Z|$ is integrable in $\{|Z'-Z|<1\}$,    we have
\begin{align*}
I_1&\leq \frac{1}{2\pi}\int_{\bR^2_+\cap \{|Z'-Z|<1\}} ( \ln (1+2y)-\ln|Z'-Z|)f(Z')\ud Z'\\&
\leq  \frac{\ln (1+2y)}{2\pi} \int_{\bR^2_+\cap \{|Z'-Z|<1\}} f(Z')\ud Z'+C. 
\end{align*}
Since $y<2R$, we get
\begin{align}\label{I1est}
I_1\leq C\ln R.
\end{align} 
Next, 
 by \eqref{Greeninfinity0}   and the fact that $g$ has support in $[0,R]$, we have 
 \begin{align}\label{I2est}
I_2&\leq C\ln(1+2y) \int_0^R\ud y'\int_{\R}\partial_x\phi(x',y')\,\ud x'\leq CR\ln(1+2y)\leq CR\ln R, 
\end{align} 
where we used again that  $y<2R$, and that  $\phi(\infty,y')=1$,  $\phi(-\infty,y')=0$.
 

From \eqref{qsplitlemmaqest}, \eqref{I1est} and \eqref{I2est}, we obtain   estimate  \eqref{qestimates}.

Next, we compute
\begin{equation}\label{Greenderesti}\partial_y G(Z',Z)=\frac{1}{2\pi}\left(\frac{y'+y}{|Z'-\tilde Z|^2}+\frac{y'-y}{|Z'- Z|^2}\right),
\quad \partial_x G(Z',Z)=\frac{1}{2\pi}\left(\frac{x-x'}{|Z'-\tilde Z|^2}-\frac{x-x'}{|Z'- Z|^2}\right). 
\end{equation}
 Therefore, using \eqref{asymptioticphi_xg}, we get
 \begin{align*}
 |\partial_y q(Z)&| =\left|\int_{\bR^2_+}\partial_yG(Z',Z)f(Z')\,\ud Z'\right| \\&
 \leq C\int_{\bR^2_+\cap\{0<y'<R\}} \left(\frac{y'+y}{(x'-x)^2+(y'+y)^2}+\frac{|y'-y|}{(x'-x)^2+(y'-y)^2}\right)\frac{1}{y'+1} \ud x'\ud y'\\&
 =C\int_0^R\ud y'\,\frac{1}{y'+1}\int_{\R} \left(\frac{y'+y}{(x'-x)^2+(y'+y)^2}+\frac{|y'-y|}{(x'-x)^2+(y'-y)^2}\right)\ud x'\\&
 =C\int_0^R\frac{1}{y'+1}\left(\arctan\left(\frac{x'-x}{y'+y}\right)\bigg\vert_{x'=-\infty}^{x'=\infty}+\arctan\left(\frac{x'-x}{|y'-y|}\right)\bigg\vert_{x'=-\infty}^{x'=\infty}\right)\,\ud y'\\&
 =C\int_0^R\frac{1}{y'+1}\,\ud y'\leq C\ln R,
  \end{align*}
 which proves \eqref{qderestimates} for $\partial_y q$.

 To estimate  $\partial_x q$, using \eqref{asymptioticphi_xg} and performing an integration by parts we get, 
 \begin{align*}
 |\partial_x q(Z)|& =\left|\int_{\bR^2_+}\partial_xG(Z',Z)f(Z')\,\ud Z'\right| \\&
 \leq C\int_{\bR^2_+\cap\{0<y'<R\}} \left(\frac{|x'-x|}{(x'-x)^2+(y'+y)^2}+\frac{|x'-x|}{(x'-x)^2+(y'-y)^2}\right)\frac{y'+1}{x'^2+(y'+1)^2}\,\ud x'\ud y'\\&
 = C\int_{\bR}\ud x'\int_{0}^R \left(\frac{|x'-x|}{(x'-x)^2+(y'+y)^2}+\frac{|x'-x|}{(x'-x)^2+(y'-y)^2}\right)\frac{y'+1}{x'^2+(y'+1)^2}\,\ud y'\\&
 = - C\int_{\bR}\ud x'\int_{0}^R \left(\arctan\left(\frac{y'+y}{|x'-x|}\right)+\arctan\left(\frac{y'-y}{|x'-x|}\right)\right)\partial_{y'}\left(\frac{y'+1}{x'^2+(y'+1)^2}\right)\,\ud y'\\&
 +C\int_{\bR}\left(\arctan\left(\frac{y'+y}{|x'-x|}\right)+\arctan\left(\frac{y'-y}{|x'-x|}\right)\right)\frac{y'+1}{x'^2+(y'+1)^2}\bigg\vert_{y'=0}^{y'=R}\,\ud x'\\&
 \leq C\int_{0}^R \ud y' \int_{\bR}\frac{1}{x'^2+(y'+1)^2}\,\ud x'+C\int_{\bR}\left(\frac{R+1}{x'^2+(R+1)^2}+\frac{1}{x'^2+1}\right)\ud x'\\&
 =C\int_{0}^R \frac{1}{y'+1}\arctan\left(\frac{x'}{y'+1}\right)\bigg\vert_{x'=-\infty}^{x'=\infty}\ud y'+
 C\left(\arctan\left(\frac{x'}{R+1}\right) +\arctan x'\right)\bigg\vert_{x'=-\infty}^{x'=\infty}\\&
 =C\int_{0}^R \frac{\ud y'}{1+y'}
 +C\leq C\ln R.
   \end{align*}
 Estimate \eqref{qderestimates} for $\partial_x q$ is then proven.

Finally, assume $y>2R$. Then, for $0<y'<R$, we have that 
\begin{equation}\label{qlemmaZ-Z'ylarge}
|Z-Z'|\geq y-y'\geq \frac{y}{2},
\end{equation} 
from which
\begin{align*}
G(Z',Z)&=\frac{1}{4\pi}\left[\ln((x'-x)^2+(y'+y)^2)-\ln((x'-x)^2+(y'-y)^2)\right]=\frac{1}{4\pi}\int_{(y'-y)^2}^{(y'+y)^2}\frac{\ud\tau}{(x'-x)^2+\tau}\\&
\leq \frac{1}{\pi}\frac{y'y}{(x'-x)^2+(y'-y)^2}  =\frac{1}{\pi}\frac{y'y}{|Z'-Z|^2}\leq \frac{1}{\pi}\frac{Ry}{|Z'-Z|^2}\leq  \frac{4R}{\pi y}.
\end{align*}
Therefore,  we get
\begin{align*}
q(Z)\leq \frac{CR}{ y}\int_0^R\ud y'\int_{\bR}\partial_x\phi(x',y')\ud x'\leq \frac{C R^2}{y},
\end{align*}
where we used again that  $\phi(\infty,y')=1$,  $\phi(-\infty,y')=0$. This gives \eqref{qderestimatesylarge} for $q$. 

Next,  recalling \eqref{Greenderesti}, by  \eqref{qlemmaZ-Z'ylarge}, for $y>2R$ and $0<y'<R$, 
$$| \partial_y G(Z',Z)|\leq\frac{1}{\pi|Z'-Z|}\leq \frac{2}{\pi y}.$$ 
Therefore, 
\begin{align*}| \partial_y q(Z)|\leq \frac{C }{y}\int_0^R\ud y'\int_{\bR}\partial_x\phi(x',y')\ud x'\leq \frac{C R}{y},
\end{align*}
which proves \eqref{qderestimatesylarge} for $\pt_yq$. The estimate for  $\pt_xq$ follows similarly. 
The proof of the lemma is then completed.

\end{proof}

\section{Construction of  super- and subsolutions to \eqref{maineq}  with \eqref{initialcondition}}\label{sec3}
 In this section, we construct  a supersolution $w_\eps$ and a subsolution $h_\eps$ to \eqref{maineq} with initial condition \eqref{initialcondition} based on   multilayers of transition profiles and appropriate correctors,  whose centers solve slightly perturbed ODE systems. In Propositions \ref{supersolutionprop} and \ref{initialconditionprop}, we will prove $w_\eps(x,y,t)$ is a supersolution to \eqref{maineq}  with initial datum \eqref{initialcondition}. The subsolution result will be summarized  in Proposition \ref{subsolutionprop}.

Consider the perturbed ODE system, 
for $i=1,\ldots, N$, $\delta>0$, 
\begin{equation} \label{PNperturbed} 
   \left\{ \begin{aligned}
     &\frac{d\bar z_i}{dt} = \frac{c_0}{\pi} \left(\sum_{ j \neq i}\frac{1}{\bar z_i - \bar z_j}-\delta\right),
     && t>0; \\
     &\bar z_i(0) = z_i^0-\delta, 
   \end{aligned} \right.
\end{equation}
and let 
\begin{equation}\label{c(t)def}
\bar c_i(t):=\dot {\bar z}_i(t), \quad \tilde\delta= \frac{\delta}{\alpha},
\end{equation}
with $\alpha$ defined in \eqref{eq:c0-gamma}. Define
\begin{equation}\label{vep}
v_\eps(x,y,t):= \sum_{i=1}^N \left[ \phi\bbs{ \frac{x-\bar z_i(t)}{\eps}, \frac{y}{\eps} } - \eps \bar c_i(t)\psi\bbs{ \frac{x-\bar z_i(t)}{\eps}, \frac{y}{\eps} }\right] + \eps \tilde{\delta}.
\end{equation}

\begin{lem}\cite[Proposition 4.3]{Mon1}\label{Monneaulem}
There exist $\eps_0,\,\delta_0>0$ such that, for any $\eps<\eps_0$,  if $(\bar z_1(t),\ldots,\bar z_N(t))$ is the solution of \eqref{PNperturbed} 
with $0<\delta<\delta_0$, then 
the function $v_\eps$ defined in \eqref{vep}
solves
\begin{equation*}\begin{cases}
-\Delta v_\eps=0,&y>0,\\
\eps \pt_t v_\eps -  \pt_y v_\eps + \frac{1}{\eps}W'(v_\eps) \geq\frac{\delta}{2},&y=0. 
\end{cases}
\end{equation*}
\end{lem}
We now introduce a new function $w_\eps$ which is obtained by adding a correction to the function $v_\eps$.
More precisely, let   $q$ be given by Lemma \ref{qlemma} where we  choose $R=2\eps^{-b}$, with $0<b<1$, 
and 
 $g$ to be a smooth nonnegative cut-off function such that
 \begin{equation}\label{cutoff}
  g(y)
=\left\{
\begin{array}{lll}
1, & \text{ for } 0\leq y\leq \frac{R}{2}=\eps^{-b};\\
0, & \text { for } y\geq  R=2\eps^{-b}.
\end{array} 
\right.   
 \end{equation}
For $\tau>0,\,\theta>0$ and $0<\gamma<1$,  define
\begin{equation}\label{w_ep}
w_\eps(x,y,t):=v_\eps(x,y,t)+\eps^{a+1} \sum_{i=1}^N\bar  c_i(t)q\bbs{ \frac{x-\bar z_i(t)}{\eps} , \frac{y}{\eps}}+\eps^\theta(y+\eps)^\gamma+\eps^{1+\tau} t.
\end{equation}

\begin{prop}\label{supersolutionprop}
Given $T>0$   and $a>0$, there exist $\eps_0,\,\delta_0,\, \tau,\,\theta>0$ and $0<b,\,\gamma<1$ such that, for any $ 0<\eps<\eps_0$,  if $(\bar z_1(t),\ldots,\bar z_N(t))$ is the solution of \eqref{PNperturbed} 
with $0<\delta<\delta_0$, then 
the function $w_\eps$ defined in \eqref{w_ep} solves
\begin{equation}\begin{cases}\label{supersolutionequations}
\eps^a\pt_t w_\eps -\Delta w_\eps\geq 0,&y>0,\, t\in (0,T);\\
\eps \pt_t w_\eps -  \pt_y w_\eps + \frac{1}{\eps}W'(w_\eps) \geq0,&y=0,\, t\in (0,T). 
\end{cases}
\end{equation}
\end{prop}
\begin{proof}
For convenience, we use the following notation throughout the proof:
$$h_i:=h\bbs{ \frac{x-\bar z_i(t)}{\eps}, \frac{y}{\eps} }$$
with $h=\phi,\,\partial_x\phi,\, \partial_y\phi, \psi,\,\partial_x\psi,\, \partial_y\psi,\, q,\,\partial_x q,\, \partial_y q$. 

Let $ 0<\eps<\eps_0$ and $ 0<\delta<\delta_0$ with $\eps_0$ and $\delta_0$ given by Lemma \ref{Monneaulem}. 
Note that, by Lemma \ref{cdotlemma}, for $\delta_0$ small enough, we have that $\bar c_i$ and $\dot{\bar c}_i$ are bounded in $[0,T]$. 

At  $y=0$,  recalling that $q(x,0)=0$, we also have $\partial_x q(x,0)=0$,  therefore, 
\begin{align*}
\eps \pt_t w_\eps -  \pt_y w_\eps + \frac{1}{\eps}W'(w_\eps) =&\eps \pt_t v_\eps +\eps^{2+\tau}
- \pt_y v_\eps -\eps^{a}\sum_{i=1}^N \bar{c}_i\partial_y q_i-\gamma\eps^{\theta+\gamma-1}
+\frac{1}{\eps}W'(v_\eps+\eps^{\theta+\gamma}+\eps^{1+\tau}t)\\
=&\eps \pt_t v_\eps -  \pt_y v_\eps + \frac{1}{\eps}W'(v_\eps)-\eps^{a}\sum_{i=1}^N \bar{c}_i\partial_y q_i+O(\eps^{\theta+\gamma-1})+O(\eps^\tau T). 
\end{align*}
By Lemma \ref{Monneaulem} and   \eqref{qderestimates}  with $R=2\eps^{-b}$, and by eventually making $\eps_0$ smaller, we get, for $\eps<\eps_0$, 
\begin{align*}
\eps \pt_t w_\eps -  \pt_y w_\eps + \frac{1}{\eps}W'(w_\eps)\geq \frac{\delta}{2}+O(\eps^{a}|\ln \eps|)+ O(\eps^{\theta+\gamma-1})+O(\eps^\tau T)\geq 0,
\end{align*}
provided
\begin{equation}\label{coeffiestim:eq1}
\theta+\gamma>1. 
\end{equation}


For $y>0$, using that $\phi$, $\psi$ and $q$  satisfy respectively \eqref{eq:standing wave}, \eqref{psi} and \eqref{qequation}, we get

\begin{equation}\label{eq_bulk1}
\begin{aligned}
\eps^a\pt_t w_\eps -  \Delta w_\eps= &-\eps^{a-1}\sum_{i=1}^N\bar c_i(t)\pt_x \phi_i+\eps^a \sum_{i=1}^N{\bar c}^2_i(t)\pt_x \psi_i-
\eps^{a+1}\sum_{i=1}^N \dot{\bar c}_i(t)\psi_i\\&
-\eps^{2a} \sum_{i=1}^N\bar{c}^2_i(t) \pt_x q_i+\eps^{2a+1}\sum_{i=1}^N\dot{\bar c}_i(t) q_i
 +\eps^{a-1}g\sum_{i=1}^N\bar c_i(t)\pt_x \phi_i \\&
 +   \eps^{a+1+\tau}+\gamma(1-\gamma)\eps^\theta(y+\eps)^{\gamma-2}.
 \end{aligned}
\end{equation}

The proof of \eqref{supersolutionequations} for $y>0$ is broken  into four cases.

\medskip

\noindent{\em Case 1: $0<y\leq \frac{\eps  R}{2}=\eps^{1-b}$.}

In this first case, by \eqref{cutoff}  $g=g(\frac{y}{\eps})=1$, so that \eqref{eq_bulk1} reads
\begin{align*}
\eps^a\pt_t w_\eps -  \Delta w_\eps= &\eps^a \sum_{i=1}^N{\bar c}^2_i(t)\pt_x \psi_i-
\eps^{a+1}\sum_{i=1}^N \dot{\bar c}_i(t)\psi_i
-\eps^{2a} \sum_{i=1}^N\bar{c}^2_i(t) \pt_x q_i\\&+\eps^{2a+1}\sum_{i=1}^N\dot{\bar c}_i(t) q_i
 +   \eps^{1+a+\tau}+\gamma(1-\gamma)\eps^\theta(y+\eps)^{\gamma-2}.
 \end{align*}
By \eqref{qestimates} and \eqref{qderestimates} with $R=2\eps^{-b}$, we have 
\begin{equation}\label{qestimatessuperprop}
0\leq q_i\leq C\eps^{-b}|\ln\eps|\quad\text{and}\quad |\partial_xq_i|\leq C|\ln\eps|,
\end{equation}
from which, recalling that $b<1$, 
\begin{align*}
\eps^a\pt_t w_\eps -  \Delta w_\eps&\geq O(\eps^a)+O(\eps^{2a}|\ln\eps|)+O(\eps^{2a+1-b}|\ln\eps|)+\gamma(1-\gamma)\eps^\theta(y+\eps)^{\gamma-2}\\&
\geq O(\eps^a)+\gamma(1-\gamma)\eps^\theta\geq 0, 
 \end{align*}
for $\eps$ small enough, provided 
\begin{equation}\label{coeffiestim:eq2}
0<\theta <a.  
\end{equation}

\medskip

  Next, for any given $b<1$, let $k_0$ be the first  integer such that 

\begin{equation}\label{k_0}
1-(k_0+1)b\leq 0.
\end{equation}
Notice that $k_0\geq 1$. 

\noindent{\em Case 2: $\eps^{1-kb}\leq y\leq\eps^{1-(k+1)b}$, $k=1,\ldots, k_0$. }

By \eqref{eq:asymptotics for phi_x} and \eqref{psiinfinityy>0}, 
$$0\leq \partial_x \phi_i,\, |\psi_i|,\,|\pt_x\psi_i|\leq \frac{C\eps}{y}\leq C\eps^{kb}.$$
Combining these estimates with \eqref{qestimatessuperprop},   and using that, since  $1-kb>0$, 
$$2a>a-(1-kb)\quad\text{and}\quad 2a+1-b>a-(1-kb),$$
\eqref{eq_bulk1} can be estimated as
\begin{align*}
\eps^a\pt_t w_\eps -  \Delta w_\eps&\geq O(\eps^{a+kb-1})+ O(\eps^{2a}|\ln\eps|)+O(\eps^{2a+1-b}|\ln\eps|)
+\gamma(1-\gamma)\eps^\theta(y+\eps)^{\gamma-2}\\&
=O(\eps^{a+kb-1})
+\gamma(1-\gamma)\eps^\theta(y+\eps)^{\gamma-2}\\&
\geq  O(\eps^{a+kb-1}) +C\eps^{\theta-(2-\gamma)[1-(k+1)b]}\geq0
 \end{align*}
for $\eps$ small enough, provided 
\begin{equation}\label{coeffiestim:eq3}
a+kb-1>\theta-(2-\gamma)[1-(k+1)b]. 
\end{equation}

\medskip

 Next, let $k_1$ be the first  integer such that 
\begin{equation}\label{k_1} \frac{k_1+1}{2}a>1.\end{equation}
Notice that $k_1\geq 0$.

\noindent{\em Case 3:   $\eps^{-\frac{k}{2}a}\leq y\leq \eps^{-\frac{k+1}{2}a}$, $k=0,\ldots, k_1$. }

In this case, by \eqref{eq:asymptotics for phi_x} and \eqref{psiinfinityy>0}, 
$$0\leq \partial_x \phi_i,\, |\psi_i|,\,|\pt_x\psi_i|\leq \frac{C\eps}{y}\leq  C\eps^{\frac{k}{2}a+1}, $$
and by \eqref{qderestimatesylarge} with $R=2\eps^{-b}$,  
$$|q_i|\leq C\frac{R^2\eps}{y}\leq C\eps^{\frac{k}{2}a+1-2b},\quad|\partial_x q_i|\leq  C\frac{R\eps}{y}\leq C\eps^{\frac{k}{2}a+1-b}.$$
Therefore, by \eqref{eq_bulk1}, and using that $b<1$, we have 
\begin{align*}
\eps^a\pt_t w_\eps -  \Delta w_\eps&\geq O\left(\eps^{\left(1+\frac{k}{2}\right)a}\right)+ O\left(\eps^{\left(2+\frac{k}{2}\right)a+1-b}\right)
+O\left(\eps^{\left(2+\frac{k}{2}\right)a+2(1-b)}\right)
+\gamma(1-\gamma)\eps^\theta(y+\eps)^{\gamma-2}\\&
=O\left(\eps^{\left(1+\frac{k}{2}\right)a}\right)+\gamma(1-\gamma)\eps^\theta(y+\eps)^{\gamma-2}\\&
\geq O\left(\eps^{\left(1+\frac{k}{2}\right)a}\right)
+C\eps^{\theta+(2-\gamma)\frac{k+1}{2}a}\geq0,
 \end{align*}
for $\eps$ small enough,  provided 
\begin{equation}\label{coeffiestim:eq4}
\left(1+\frac{k}{2}\right)a>\theta +(2-\gamma)\frac{k+1}{2}a. 
\end{equation}

\medskip

Next, by \eqref{k_1} there exists $r>0$ so  small that 
\begin{equation}\label{rk_1} \frac{k_1+1}{2}a>1+r.
\end{equation}
Then, we are left with  the following last case.

\medskip

\noindent{\em Case 4: $y\geq \eps^{-1-r}$.}

In this case, by \eqref{eq:asymptotics for phi_x} and \eqref{psiinfinityy>0}, 
$$0\leq \partial_x \phi_i,\, |\psi_i|,\,|\pt_x\psi_i|\leq \frac{C\eps}{y}\leq C\eps^{2+r}, $$
and by \eqref{qderestimatesylarge}, 
$$|q_i|\leq    \frac{CR^2\eps}{y}\leq C\eps^{2+r-2b},\quad|\partial_x q_i|\leq    \frac{CR\eps}{y}\leq \eps^{2+r-b}.$$
Therefore, by \eqref{eq_bulk1}  and using that for $b<1$, 
$$2a-b+2+r,\,2(a-b)+3+r>a+1+r,$$ we obtain
\begin{align*}
\eps^a\pt_t w_\eps -  \Delta w_\eps&\geq  O(\eps^{a+1+r})+ O(\eps^{2a-b+2+r})+ O(\eps^{2(a-b)+3+r})+\eps^{a+1+\tau}\\&
= O(\eps^{a+1+r})+\eps^{a+1+\tau}\geq0,
 \end{align*}
for $\eps$ small enough, provided 
\begin{equation}\label{coeffiestim:eq6}
0<\tau<r. 
\end{equation}

 \medskip

Putting it all together, we first choose $b$ satisfying 
$$0<b<\min\{1,a\}.$$
Note that for $b<a$ and any integer $k$,  we have that 
$$(k+1)b-1 < a+kb-1.$$
Therefore,  since the quantity $\theta-(2-\gamma)[1-(k+1)b]$ is close to $(k+1)b-1$ when $\theta$ is close to 0 and $\gamma$ is close to 1, we can choose 
  $\theta$ sufficiently small and $\gamma$ sufficiently close to 1 so that
$$ 0<\theta<a,\quad 1-\theta<\gamma<1,$$ and 
condition \eqref{coeffiestim:eq3} is satisfied for $k=1,\ldots, k_0$ with $k_0$ as in \eqref{k_0}. 
Moreover, since the quantity $\theta +(2-\gamma)\frac{k+1}{2}a$ is close to $\frac{k+1}{2}a  < \left(1+\frac{k}{2}\right)a$, by eventually choosing $\theta$ smaller and $\gamma$ closer to 1, condition \eqref{coeffiestim:eq4} holds true for $k=0,\ldots,k_1$ with $k_1$ as in \eqref{k_1}. 
Finally, we choose $r>0$  satisfying \eqref{rk_1}  and $0<\tau<r$. 

With this choice of the coefficients, conditions \eqref{coeffiestim:eq1}, \eqref{coeffiestim:eq2}, \eqref{coeffiestim:eq3},  \eqref{coeffiestim:eq4} and \eqref{coeffiestim:eq6} are satisfied and the above computations show that $w_\eps$ is solution to  
 \eqref{supersolutionequations}, as desired. This concludes the proof of the proposition. 
\end{proof}

  We next show that the function $w_\eps$ defined in \eqref{w_ep} is above   the initial condition \eqref{initialcondition} at initial time. 
 
 \begin{prop}\label{initialconditionprop}
There exist $\eps_0,\,\delta_0>0$ such that for any $0<\eps<\eps_0$, if $(\bar z_1(0),\ldots,\bar z_N(0))$ satisfies \eqref{PNperturbed} with $0<\delta<\delta_0$,   and $b$ is as in Proposition \ref{supersolutionprop},  then  the function  $w_\eps$ defined in \eqref{w_ep} satisfies
\begin{equation}\label{initialconditionwep}
w_\eps(x,y,0)\geq  u_\eps^0(x,y)\quad\text{for all }(x,y)\in \overline{\R^2_+}, 
\end{equation}
with $u_\eps^0$ defined in \eqref{initialcondition}. 
\end{prop}
\begin{proof}

First note that, by the monotonicity of $\phi$ with respect to $x$, for every $i=1,\ldots, N$, 
\begin{equation}\label{phi-barphieuqpropintial}
\phi\bbs{ \frac{x-\bar z_{i}(0)}{\eps} , \frac{y}{\eps}}\geq \phi\bbs{ \frac{x-z^0_{i}}{\eps} , \frac{y}{\eps}}. 
\end{equation}
By  \eqref{psiinfinityx>1}, there exists $K>0$ such that 
\begin{equation}\label{Propinitialconpsifarcond}\sup_{|x|>K}|\psi(x,y)|\sum_{i=1}^N|\bar c_i(0)|\leq\frac{\tilde\delta}{2}. 
\end{equation}
Moreover, by \eqref{qestimates} with $R=2\eps^{-b}$,   and recalling that  $0<b<a$, for $\eps$ small enough, 
\begin{equation}\label{qestpropintial}0\leq \eps^{a+1} \sum_{i=1}^N|\bar  c_i(0)|q\bbs{ \frac{x-\bar z_i(0)}{\eps} , \frac{y}{\eps}}\leq C\eps^{a+1-b}|\ln\eps|\leq \frac{\eps \tilde\delta}{2}.
\end{equation}
Now, for  fixed $(x,y)\in  \overline{\R^2_+},$ we consider two cases. 

\medskip

\noindent{\em Case 1: there exists $i_0=1,\ldots,N$, such that $|x-\bar z_{i_0}(0)|\leq \eps K. $}

By the monotonicity of $\phi$ with respect to $x$, 
$$\phi\bbs{ \frac{x-\bar z_{i_0}(0)}{\eps} , \frac{y}{\eps}}\geq \phi\bbs{ -K,\frac{y}{\eps}},$$
while,  for $\eps$ small enough, 
$$\phi\bbs{ \frac{x-z^0_{i_0}}{\eps} , \frac{y}{\eps}}=\phi\bbs{ \frac{x-\bar z_{i_0}(0)-\delta}{\eps} , \frac{y}{\eps}}\leq \phi\bbs{ K-\frac{\delta}{\eps} , \frac{y}{\eps}} \leq 
\phi\bbs{-\frac{\delta}{2\eps} , \frac{y}{\eps}}.  $$
Therefore, from \eqref{eq:asymptotics for phi_x}, 
\begin{align*}
\phi\bbs{ \frac{x-\bar z_{i_0}(0)}{\eps} , \frac{y}{\eps}}-\phi\bbs{ \frac{x-z^0_{i_0}}{\eps} , \frac{y}{\eps}}&\geq \phi\bbs{ -K,\frac{y}{\eps}}-\phi\bbs{-\frac{\delta}{2\eps} , \frac{y}{\eps}}
=\int_{-\frac{\delta}{2\eps}}^{-K}\partial_x\phi\bbs{ \tau , \frac{y}{\eps}}\ud \tau\\&
\geq C \int_{-\frac{\delta}{2\eps}}^{-K}\frac{\frac{y}{\eps}+1}{\tau^2+\left(\frac{y}{\eps}+1\right)^2}\ud \tau\\&
\geq C\left(\frac{\delta}{2\eps}-K\right)\frac{\frac{y}{\eps}+1}{\left(\frac{\delta}{\eps}\right)^2+\left(\frac{y}{\eps}+1\right)^2}\\&
\geq C\frac{\delta}{\eps}\frac{\frac{y}{\eps}+1}{\left(\frac{\delta}{\eps}\right)^2+\left(\frac{y}{\eps}+1\right)^2}\\&
=  C\delta\frac{y+\eps}{\delta^2+(y+\eps)^2}.
\end{align*}
By \eqref{psiinfinityy>0}, for $\eps$ small enough, 
we infer that,

\begin{equation*}
\phi\bbs{ \frac{x-\bar z_{i_0}(0)}{\eps} , \frac{y}{\eps}}-\phi\bbs{ \frac{x-z^0_{i_0}}{\eps} , \frac{y}{\eps}}\geq  \eps \left| \bar c_{i_0}(0)\psi\bbs{ \frac{x-\bar z_{i_0}(0)}{\eps} , \frac{y}{\eps}}\right|.
\end{equation*}
Next, since $|x-\bar z_{i_0}(0)|<\eps K$, we have that $|x-\bar z_i(0)|>\eps K$  for $i\neq i_0$  and by \eqref{Propinitialconpsifarcond}, 
$$\sum_{i\neq i_0}  \left|\bar c_i(0)\psi\bbs{ \frac{x-\bar z_i(0)}{\eps}, \frac{y}{\eps} }\right|\leq\frac{\tilde\delta}{2}.$$ 
Combining the two last estimates with \eqref{phi-barphieuqpropintial} and \eqref{qestpropintial}, yields
\begin{align*}
w_\eps(x,y,0)=  &\sum_{i=1}^N \left[ \phi\bbs{ \frac{x-\bar z_i(0)}{\eps}, \frac{y}{\eps} } - \eps \bar c_i(0)\psi\bbs{ \frac{x-\bar z_i(0)}{\eps}, \frac{y}{\eps} }\right] + \eps \tilde{\delta}
\\&+\eps^{a+1} \sum_{i=1}^N\bar  c_i(0)q\bbs{ \frac{x-\bar z_i(0)}{\eps} , \frac{y}{\eps}}+\eps^\theta(y+\eps)^\gamma\\
\geq & \left[ \phi\bbs{ \frac{x-\bar z_{i_0}(0)}{\eps}, \frac{y}{\eps} } - \eps \bar c_{i_0}(0)\psi\bbs{ \frac{x-\bar z_{i_0}(0)}{\eps}, \frac{y}{\eps} }\right] 
+\sum_{i\neq i_0} \phi\bbs{ \frac{x-\bar z_i(0)}{\eps}, \frac{y}{\eps} } \\&- \eps\sum_{i\neq i_0}  \bar c_i(0)\psi\bbs{ \frac{x-\bar z_i(0)}{\eps}, \frac{y}{\eps} }
  +\eps^{a+1} \sum_{i=1}^N\bar  c_i(0)q\bbs{ \frac{x-\bar z_i(0)}{\eps} , \frac{y}{\eps}}+\eps \tilde{\delta}  \\
 \geq &  \phi\bbs{ \frac{x-z^0_{i_0}}{\eps} , \frac{y}{\eps}} +\sum_{i\neq i_0} \phi\bbs{ \frac{x- z_i^0}{\eps}, \frac{y}{\eps} } -\frac{\eps\tilde\delta}{2}-\frac{\eps\tilde\delta}{2}+\eps\tilde\delta\\
= &\sum_{i=1}^N \phi\bbs{ \frac{x- z_i^0}{\eps}, \frac{y}{\eps} }=u_\eps^0(x,y),
\end{align*}
as desired.

\medskip

\noindent{\em Case 2:  $|x-\bar z_{i_0}(0)|> \eps K$, for all $i=1,\ldots,N$.  }

By  \eqref{Propinitialconpsifarcond},
$$ \sum_{i=1}^N  \left|\bar c_i(0)\psi\bbs{ \frac{x-\bar z_i(0)}{\eps}, \frac{y}{\eps} }\right|\leq\frac{\tilde\delta}{2},$$ 
which together with   \eqref{phi-barphieuqpropintial} and \eqref{qestpropintial} implies 
\begin{align*}
w_\eps(x,y,0)\geq&  \sum_{i=1}^N \phi\bbs{ \frac{x-\bar z_i(0)}{\eps}, \frac{y}{\eps} }- \eps\sum_{i=1}^N  \bar c_i(0)\psi\bbs{ \frac{x-\bar z_i(0)}{\eps}, \frac{y}{\eps} }\\&
  +\eps^{a+1} \sum_{i=1}^N\bar  c_i(0)q\bbs{ \frac{x-\bar z_i(0)}{\eps} , \frac{y}{\eps}}+\eps \tilde{\delta}\\
\geq &\sum_{i=1}^N \phi\bbs{ \frac{x- z_i^0}{\eps}, \frac{y}{\eps} }=u_\eps^0(x,y). 
\end{align*}

From Cases 1 and 2, we infer that \eqref{initialconditionwep} holds for every $(x,y)\in\overline{\R^2_+}$. 
This completes the proof of the proposition. 

\end{proof}

Subsolutions to \eqref{maineq}  with initial datum \eqref{initialcondition} are constructed in a manner similar to that of supersolutions.
Consider the perturbed system, 
for $i=1,\ldots, N$, $\delta>0$, 
\begin{equation} \label{PNperturbedsub} 
   \left\{ \begin{aligned}
     &\frac{d\underline z_i}{dt} = \frac{c_0}{\pi} \left(\sum_{ j \neq i}\frac{1}{\underline  z_i - \underline z_j}+\delta\right),
     && t>0, \\
     &\underline z_i(0) = z_i^0+\delta, 
   \end{aligned} \right.
\end{equation}
 
and let 
\begin{equation*}
\underline c_i(t):=\dot { \underline z}_i(t),
\end{equation*}
and $\tilde\delta$ be defined as in \eqref{c(t)def}. Then, one can prove that the function 
\begin{equation}\label{hepsub}
\begin{aligned}
h_\eps(x,y,t):=&\sum_{i=1}^N \left[ \phi\bbs{ \frac{x-\underline z_i(t)}{\eps}, \frac{y}{\eps} } - \eps \underline c_i(t)\psi\bbs{ \frac{x-\underline z_i(t)}{\eps}, \frac{y}{\eps} }\right] - \eps 
\tilde{\delta} +\eps^{a+1} \sum_{i=1}^N\underline c_i(t)q\bbs{ \frac{x-\underline  z_i(t)}{\eps} , \frac{y}{\eps}}\\&-\eps^\theta(y+\eps)^\gamma-\eps^{1+\tau} t,
\end{aligned}
\end{equation}
is subsolution to \eqref{maineq}
 with initial datum \eqref{initialcondition}. More precisely, we have  

\begin{prop}\label{subsolutionprop}
Given $T>0$  and $a>0$, there exist $\eps_0,\,\delta_0,\, \tau,\,\theta>0$ and $0<b,\,\gamma<1$ such that, for any $0<\eps<\eps_0$,  if $(\underline z_1(t),\ldots,\underline z_N(t))$ is the solution of \eqref{PNperturbedsub} 
with $0<\delta<\delta_0$, then 
the function $h_\eps$ defined in \eqref{hepsub} solves 
\begin{equation*}\begin{cases}
\eps^a\pt_t h_\eps -\Delta h_\eps\leq 0,&y>0,\, t\in(0,T),\\
\eps \pt_t h_\eps -  \pt_y h_\eps + \frac{1}{\eps}W'(h_\eps) \leq0,&y=0, t\in(0,T), 
\end{cases}
\end{equation*}
and $$h_\eps(x,y,0)\leq  u_\eps^0(x,y)\quad\text{for all }(x,y)\in \overline{\R^2_+}. $$
\end{prop}

\section{Proof of Theorem \ref{mainthm} }\label{sec4}
In this section, we complete the proof of Theorem \ref{mainthm} using the constructed super/subsolutions,  the comparison principle and the decay estimates established in previous sections.

\begin{proof}[Proof of Theorem \ref{mainthm}] 
Let $w_\eps$ and $h_\eps$ be the functions defined in \eqref{w_ep} and \eqref{hepsub}. 
Given any $T>0$, by Propositions \ref{supersolutionprop}, \ref{initialconditionprop} and \ref{subsolutionprop} there exist $\delta_0,\,\eps_0>0$ 
  and coefficients $\theta,\,\tau>0$, $0<b,\,\gamma <1$ such that for $0<\eps<\eps_0$ and $0<\delta<\delta_0$,  $w_\eps$ and $h_\eps$ are respectively super and subsolution to \eqref{maineq}  with initial datum \eqref{initialcondition} in $\R^2_+\times[0,T]$.
Since $w_\eps$ and $h_\eps$ are strictly sublinear in $y$, we can apply the comparison principle to conclude that 
\begin{equation}\label{uepinbetween}
h_\eps(x,y,t)\leq \ue(x,y,t)\leq w_\eps(x,y,t)\quad\text{for all }(x,y,t)\in \overline{\R^2_+}\times[0,T]. 
\end{equation}
Note that from \eqref{qestimates} with $R=2\eps^{-b}$, $0<b<1$,  
$$\eps^{a+1} \sum_{i=1}^N\bar  c_i(t)q\bbs{ \frac{x-\bar z_i(t)}{\eps} , \frac{y}{\eps}},\, \eps^{a+1} \sum_{i=1}^N\bar  c_i(t)q\bbs{ \frac{x-\underline z_i(t)}{\eps} , \frac{y}{\eps}}
\to 0\quad \text{as }\eps\to 0^+.
$$
Let $(x,y,t)\in \overline{\R^2_+}\times[0,\infty)$. Then,  from \eqref{uepinbetween}, 
\begin{equation}\label{limsupmainthm}
 \limsup_{\eps\to0^+}\ue(x,y,t)\leq   \limsup_{\eps\to0^+}\sum_{i=1}^N \phi\bbs{ \frac{x-\bar z_i(t)}{\eps}, \frac{y}{\eps}},
\end{equation}
and 
\begin{equation}\label{liminfmainthm}
\liminf_{\eps\to0^+}\ue(x,y,t)\geq \liminf_{\eps\to0^+}\sum_{i=1}^N \phi\bbs{ \frac{x-\underline z_i(t)}{\eps}, \frac{y}{\eps}},
\end{equation}
since the other terms in $w_\eps$ and $h_\eps$ vanish when $\eps$ goes to 0. 
From  \eqref{phiasymptocsylarge}, when $y>0$, 
$$ \limsup_{\eps\to0^+}\ue(x,y,t)\leq \frac{1}{\pi}\sum_{i=1}^N \left(\frac{\pi}{2}+\arctan\left(\frac{x-\bar z_i(t)}{y}\right)\right)$$
and 
$$\liminf_{\eps\to0^+}\ue(x,y,t)\geq \frac{1}{\pi}\sum_{i=1}^N \left(\frac{\pi}{2}+\arctan\left(\frac{x-\underline z_i(t)}{y}\right)\right).$$
Sending $\delta\to 0$ yields Statement (ii) of   Theorem \ref{mainthm}. 

Statement (i) follows from  \eqref{limsupmainthm}, \eqref{liminfmainthm} and  the following Lemma \ref{lem4.1}. This completes the proof of  Theorem \ref{mainthm}.
\end{proof} 
\begin{lem}\label{lem4.1}
 Let $v_0$ be defined as \eqref{v0def}, $\phi$ be the stationary layer solution solving \eqref{eq:standing wave}, and $\bar z_i,\, \underline z_i$ solve ODEs \eqref{PNperturbed}, \eqref{PNperturbedsub} respectively. Then we have
\begin{equation}\label{limsupmainthmphi}
 \limsup_{\delta\to0^+}\limsup_{(x',y',t')\to(x,0,t)\atop\eps\to0^+} \sum_{i=1}^N \phi\bbs{ \frac{x'-\bar z_i(t')}{\eps}, \frac{y'}{\eps}}\leq( v_0)^*(x,t),
\end{equation} and
\begin{equation}\label{liminfmainthmphi}
 \liminf_{\delta\to0^+}\liminf_{(x',y',t')\to(x,0,t)\atop\eps\to0^+} \sum_{i=1}^N \phi\bbs{ \frac{x'-\underline z_i(t')}{\eps}, \frac{y'}{\eps}}\geq (v_0)_*(x,t).
\end{equation} 
\end{lem}

\begin{proof}
  
Let us prove \eqref{limsupmainthmphi}. The proof of \eqref{liminfmainthmphi} follows with a similar argument. 
Let $H^*$ be defined by $H^*(s)=H(s)$ if $s\neq 0$ and $H^*(0)=1$. It is easy to prove that 
$$(v_0)^*(x,t)= \sum_{i=1}^N H^*(x-z_i(t)).$$
Fix $(x,t)$ and consider two cases.

\medskip

\noindent{\em Case 1: There exists $i_0$ such that $x=z_{i_0}(t)$. }

Let $(x_n,y_n,t_n)$ be a sequence converging to $(x,  0,t)$.  By Lemma \ref{cdotlemma}, for $\delta$ small enough,
we have that 
$$x-\bar{z}_i(t)>0\text{ for }i=1,\ldots,i_0-1\quad\text{and}\quad x-\bar{z}_i(t)<0\text{ for }i=i_0+1,\ldots N, $$
and for $\eps$  small enough and $n$ large enough, 
$$x_n-\bar{z}_i(t_n)\geq \eps^\frac14\text{ for }i=1,\ldots,i_0-1\quad\text{and}\quad x_n-\bar{z}_i(t_n)\leq -\eps^\frac14\text{ for }i=i_0+1,\ldots N.  $$
Assume that $y_n=0$,   for all $n$. Then, by the monotonicity of $\phi$ and its behavior at infinity we get
$$\lim_{n\to\infty\atop\eps\to0^+} \phi\bbs{ \frac{x_n-\bar z_i(t_n)}{\eps}, 0}=\begin{cases} 1,&\text{if }i=1,\ldots, i_0-1\\
0,&\text{if }i=i_0+1,\ldots, N
\end{cases}=H^*(x-\bar z_i(t)).
$$
 By \eqref{phiasymptocsylarge}, the same limit holds true when $y_n>0$   and $y_n\to 0$.  

When $i=i_0$, since $\phi<1$ we simply have
$$\limsup_{n\to\infty \atop\eps\to0^+} \phi\bbs{ \frac{x_n-\bar z_{i_0}(t_n)}{\eps}, \frac{y_n}{\eps}}\leq 1=H^*(x-z_{i_0}(t)).$$
Since the limits above are computed along any arbitrary sequence  $(x_n,y_n,t_n)$  converging to $(x,  0,t)$, we conclude that 
\begin{align*}
 \limsup_{\delta\to0^+}\limsup_{(x',y',t')\to(x,0,t)\atop\eps\to0^+} \sum_{i=1}^N \phi\bbs{ \frac{x'-\bar z_i(t')}{\eps}, \frac{y'}{\eps}}&\leq
  \limsup_{\delta\to0^+}\sum_{i=1}^N\limsup_{(x',y',t')\to(x,0,t)\atop\eps\to0^+ }\phi\bbs{ \frac{x'-\bar z_i(t')}{\eps}, \frac{y'}{\eps}}\\&
  \leq  \limsup_{\delta\to0^+}\sum_{i\neq i_0}^N H^*(x-\bar z_i(t))+H^*(x-z_{i_0}(t))\\&
  =(v_0)^*(x,t).
\end{align*}
This proves  \eqref{limsupmainthmphi} in Case 1. 

\medskip

\noindent{\em Case 2: $x\neq z_{i}(t)$, for all $i=1,\ldots, N.$ }

Arguing as in Case 1, we obtain that, for all $i=1,\ldots, N$, 
$$\limsup_{(x',y',t')\to(x,0,t)\atop\eps\to0^+} \phi\bbs{ \frac{x'-\bar z_i(t')}{\eps}, \frac{y'}{\eps}}=H^*(x-\bar z_i(t)),$$
from which \eqref{limsupmainthmphi} follows. 
\end{proof}

\section*{Acknowledgment}
The first author has been partially supported by NSF Grant DMS-2204288 ``Bulk-Interface Coupled Response in Novel Materials: Pattern Formation and Interactive Migration".

The second author has been supported by the NSF Grant DMS-2155156 ``Nonlinear PDE methods in the study of interphases.''

\bibliographystyle{alpha}
\bibliography{Dislocation_bib}

\end{document}